\documentclass[smallextended,nospthms]{svjour3}
\RequirePackage{fix-cm}
\usepackage{amssymb}
\usepackage{amsmath}
\usepackage{amsfonts}
\usepackage{color}
\usepackage{algorithm}
\usepackage{algpseudocode}
\usepackage{hyperref}
\usepackage{graphicx}
\usepackage{graphics}
\usepackage[numbers]{natbib}
\newtheorem{Proposition}{Proposition}[section]
\newtheorem{Theorem}{Theorem}[section]
\newtheorem{Remark}{Remark}[section]

\newtheorem{Lemma}{Lemma}[section]
\newtheorem{Definition}{Definition}[section]

\def\qed{\hbox to 0pt{}\hfill$\rlap{$\sqcap$}\sqcup$}
\newenvironment{proof}[1][Proof]{\noindent\textbf{#1.} }{\ \rule{0.5em}{0.5em}}

\journalname{Computational Optimization and Applications}

\begin{document}

\title{  
 A multi-criteria approach to approximate solution of multiple-choice knapsack problem}

\author{Ewa M. Bednarczuk \and Janusz Miroforidis \and Przemys\l aw Pyzel}

\institute{Ewa M. Bednarczuk 
	\at Systems Research Institute, Polish Academy of Sciences \\
	ul. Newelska 6, 01-447 Warszawa, Poland \\
	Warsaw University of Technology \\
	Faculty of Mathematics and Information Science \\
	ul. Koszykowa 75, 00-662 Warszawa, Poland \\
	 \email{Ewa.Bednarczuk@ibspan.waw.pl} \\
	\and Janusz Miroforidis
	\at Systems Research Institute, Polish Academy of Sciences\\ 
	ul. Newelska 6, 01-447 Warszawa, Poland \\ \email{Janusz.Miroforidis@ibspan.waw.pl} \\
	\and 
	Przemys\l aw Pyzel
	\at PhD Programme Systems Research Institute, Polish Academy of Sciences \\ 
	ul. Newelska 6, 01-447 Warszawa, Poland \\
	\email{PPyzel@ibspan.waw.pl}
}

	
\maketitle

\begin{abstract}
We propose a method for finding approximate  solutions to multiple-choice knapsack problems.  To this aim we transform the multiple-choice knapsack problem into a bi-objective optimization problem  whose solution set  contains solutions of the original multiple-choice knapsack problem. The method relies on solving a series of suitably defined  linearly scalarized  bi-objective problems. The  novelty which makes the method attractive from the computational point of view is that we are able to solve explicitly those linearly scalarized  bi-objective problems   with the help of the closed-form formulae.

The method is computationally analyzed on a set of large-scale problem  instances   (test problems) of two categories: uncorrelated and weakly correlated. Computational results show that after solving, in average 10 scalarized  bi-objective  problems, the optimal value of the original knapsack problem  is approximated  with the accuracy  comparable to the accuracies  obtained by the greedy algorithm and an exact algorithm. More importantly, the respective approximate solution to the original knapsack problem  (for which the approximate optimal value is attained) can be found without resorting to the dynamic programming. In the test problems, the number of multiple-choice constraints ranges up to hundreds with hundreds variables in each constraint.
\end{abstract}

\section{Introduction}
\label{intro} 

The multi-dimensional multiple-choice knapsack problem $(MMCKP)$ and 
the multiple-choice knapsack problem $(MCKP)$
are classical generalizations of the knapsack problem $(KP)$ and  are applied to modeling many real-life problems, e.g., in  project (investments) portfolio selection   \cite{Nauss1978,Zhong2010}, capital budgeting \cite{Pisinger2001}, advertising  \cite{Sinha1979}, component selection in IT systems  \cite{Kwong2010,Pyzel2012}, computer networks management \cite{Lee1999}, adaptive multimedia systems \cite{Khan1998}, and other. 

The {\em multiple-choice knapsack problem} $(MCKP)$ is formulated as follows. 
Given  are $k$ sets $N_{1}$, $N_{2}$,...,$N_{k}$ of items,
of cardinality $|N_{i}|=n_{i}$, $i=1,...,k$.
Each item of each set
has  been assigned real-valued nonnegative 'profit' $p_{ij}\ge 0$ and 'cost'  $c_{ij}\ge 0$,
$i=1,...,k$, $j=1,...,n_{i}$.

The problem consists in choosing 
exactly one item from each set $N_{i}$ so that
the total cost does not exceed a 
given $b\ge 0$ 
and the total  profit is maximized.

Let $x_{ij}$,
$i=1,...,k$, $j=1,...,n_{i}$,  be
defined as
$$
x_{ij}=\left\{\begin{array}{ll}
1&\mbox{if item $j$ from set $N_{i}$ is chosen}\\
0&\mbox{otherwise.}
\end{array}\right.
$$
 Note that all $x_{ij}$ form a vector $x$ of length $n=\sum_{i=1}^{k}n_{i}$, 
$x\in\mathbb{R}^{n}$, and we write 
$$
\begin{array}{l}
x:=(x_{11},x_{12},...,x_{1n_{1}},x_{21},...,
x_{2n_{2}},...., x_{k1},x_{k2},...x_{kn_{k}})^{T}.\\
\end{array}
$$
In this paper, we adopt the convention that a vector $x$ is a column vector, and hence  the
transpose of $x$, denoted by $x^{T}$, is a row vector.
 
Problem $(MCKP)$ is of the form
$$
\begin{array}{ll}
(MCKP)&\begin{array}{l}\max\ \ \sum_{i=1}^{k}\sum_{j=1}^{n_{i}}p_{ij}x_{ij}\\
\mbox{subject to}\\
\sum_{i=1}^{k}\sum_{j=1}^{n_{i}}c_{ij}x_{ij}\le b\\
(x_{ij})\in X:=\{(x_{ij})\ |\ \sum_{j=1}^{n_{i}}x_{ij}=1, \\
x_{ij}\in\{0,1\}\ \ i=1,...,k, \ \ j=1,...,n_{i}\}.
\end{array}
\end{array}
$$
{ By using the above notations, problem $(MCKP)$ can be equivalently rewritten in  the vector form
	$$
	\begin{array}{ll}
	(MCKP)&\begin{array}{l}\max\ \ p^{T} x\\
	\mbox{subject to}\\
	c^{T} x\le b\\
	x=(x_{ij})\in X,
	\end{array}
	\end{array}
	$$
	where $p$ and $c$ are vectors from $\mathbb{R}^{n}$,
	$$
	\begin{array}{l}
	p:=(p_{11},p_{12},...,p_{1n_{1}},p_{21},...,
	p_{2n_{2}},...., p_{k1},p_{k2},...p_{kn_{k}})^{T}\\
	c:=(c_{11},c_{12},...,c_{1n_{1}},c_{21},...,
	c_{2n_{2}},...., c_{k1},c_{k2},...c_{kn_{k}})^{T},\\
	\end{array}
	$$
	and for any vectors $u, v\in\mathbb{R}^{n}$, the scalar product $u^{T}v$ is defined in the usual way as $u^{T}v:=\sum_{i=1}^{n}u_{i}v_{i}$.

	 	The feasible set $F$  to problem $(MCKP)$ is defined by a single linear inequality constraint and the constraint $x\in X$, i.e.,
	 	$$
	 	F:=\{x\in\mathbb{R}^{n}\ | \ c^{T}x\le b,\ \ x\in X\}
	 	$$
	 and finally
	 $$
	 \begin{array}{ll}
	 (MCKP)&\begin{array}{l}\max\ \ p^{T} x\\
	 \mbox{subject to}\\
	 x\in F.
	 \end{array}
	 \end{array}
	 $$
	 The optimal value of problem $(MCKP)$ is equal to $\max_{x\in F} \ p^{T} x$ 
	 and the solution set
	 $S^{*}$ is given as
	 $$
	 S^{*}:=\{\bar{x}\in F\ |\ 	p^{T}\bar{x}=\max_{x\in F} \ p^{T} x\}.
	 $$

Problem $(MCKP)$  is ${\cal N}{\cal P}$-hard. The approaches to solving $(MCKP)$ can be: heuristics \cite{Akbar2006,Hifi2004}, exact methods providing upper bounds for the optimal value of the profit together with the corresponding approximate solutions \cite{Sbihi2007}, exact methods providing solutions \cite{Martello2000}. There are algorithms that efficiently solve $(MCKP)$  without sorting and reduction \cite{Dyer1984,Zemel1984} or with sorting and reduction \cite{Dudzinski1984}. Solving $(MCKP)$ with a linear relaxation (by neglecting the constrains
$x_{ij}\in\{0,1\},\ \ i=1,...,k, \ \ j=1,...,n_{i}$) gives  upper bounds on the value of optimal profit. Upper bounds can be also obtained with the help of the Lagrange relaxation. 
These facts and other features of $(MCKP)$ are described in details in monographs \cite{Kellerer2004,Martello1990}.

Exact  branch-and-bound methods \cite{Dyer1984a} (integer programming), even those using commercial optimization software (e.g., LINGO, CPLEX) can have troubles with solving large $(MCKP)$ problems. 
A branch-and-bound algorithm with a quick solution of the relaxation of reduced problems was proposed by Sinha and Zoltners \cite{Sinha1979}. Dudzi\'nski and Walukiewicz proposed an algorithm with pseudo-polynomial complexity \cite{Dudzinski1987}. 

Algorithms that use dynamic programming require integer values of data and for large-scale problems require large amount of memory for backtracking (finding solutions in  set $X$), see also the monograph \cite{Martello1990}. The algorithm we propose does not need the data to be integer numbers. 

Heuristic algorithms, based on solving linear (or continuous) relaxation of $(MCKP)$ and dynamic programming \cite{Dyer1995,Pisinger1995,Pisinger2001} are reported to be fast, but have limitations typical for dynamic programming. 

The most recent approach "reduce and solve" \cite{Chen2014,Gao2016} is based on reducing the problem by proposed pseudo cuts and then solving the reduced problems by a Mixed Integer  Programming $(MIP)$ solver.

In the present paper, we  propose a new exact (not heuristic) method 
which provides approximate optimal profits  
together with the corresponding approximate solutions. 
The method is based on multi-objective optimization
 techniques. Namely, we start by formulating a linear bi-objective problem $(BP)$ related to the original problem $(MCKP)$. After investigating the relationships
	between $(MCKP)$ and $(BP)$ problems, we propose an algorithm for solving $(MCKP)$ via
	   a series of  scalarized linear bi-objective problems $(BS(\lambda))$. 

The main advantage of the proposed method is that the scalarized linear bi-objective problems $(BS(\lambda))$ can be explicitly solved by exploiting the structure of the set $X$. Namely, these scalarized problems can be decomposed into $k$ independent subproblems 
 the solutions of which are given by simple  closed-form formulas.
 This feature of our method is particularly suitable for parallelization. It allows to 
  generate  solutions of scalarized problems in an efficient and fast way.

 The experiments show that the  method we propose generates
	 very quickly an  outcome  $\hat{x}\in F$ which is an approximate solution to $(MCKP)$. 
	 Moreover, lower bound (LB) and upper bound (UB) for the optimal profit are provided. 

  The obtained approximate solution  $\hat{x}\in F$ could  serve as a good starting point for other, e.g., heuristic or exact algorithms for finding an optimal solution to the problem $(MCKP)$.

 
The organization of the paper is as follows. In Section 2, we provide preliminary facts on multi-objective optimization problems and we formulate a bi-objective optimization problem $(BP)$ associated with $(MCKP)$.
In Section 3, we investigate the relationships between the problem $(BP)$ and the original problem $(MCKP)$. In Section 4, we formulate  scalarized problems $(BS(\lambda))$ for bi-objective problem $(BP)$ and we provide closed-form formulae for solutions to  problems $(BS(\lambda))$ by 
 decomposing them into $k$ independent subproblems $(BS(\lambda))_{i}$, $i=1,...,k$. In Section 5, we present our method (together with the pseudo-code) which provides a lower bound (LB) for the optimal profit together with the corresponding approximate feasible solution  $\hat{x}\in F$ to $(MCKP)$ for which the bound (LB) is attained. 
In Section 6, we report  on  the results of numerical experiments. The last section concludes.

\section{Multi-objective optimization problems}
\label{multi}
Let $f_{i}:\mathbb{R}^{n}\rightarrow\mathbb{R}$, $i=1,...,k$, 
be  functions defined on $\mathbb{R}^{n}$ and 
$\Omega\subset\mathbb{R}^{n}$ be a   subset in $\mathbb{R}^{n}$.

The {\em multi-objective optimization problem} is defined as
$$
\begin{array}{ll}
(P)&\begin{array}{l}
\text{V}\max\ \ (f_{1}(x),...,f_{k}(x))\\
\mbox{subject to}\\
x\in\Omega,\\
\end{array}\end{array}
$$
where the symbol $'\text{V}\max'$
means that solutions to problem $(P)$ are understood 
 in the sense of Pareto efficiency defined in
Definition \ref{pareto_sol}.

Let
$$
\mathbb{R}_{+}^{k}:=\{x=(x_{1},...,x_{k})\in\mathbb{R}^{k}\ :\ x_{i}\ge 0,\ i=1,...,k\}.
$$

\begin{Definition}
\label{pareto_sol}
A point $x^{*}\in\Omega$ is a {\em Pareto efficient (Pareto maximal) solution} 
to $(P)$ if
$$
f(\Omega)\cap (f(x^{*})+\mathbb{R}^{k}_{+})=\{f(x^{*})\}.
$$
In other words, $x^{*}\in\Omega$ is a  Pareto efficient solution to $(P)$ if
there is no $\bar{x}\in\Omega$ such that
\begin{equation}
\label{pareto}
\begin{array}{l}
f_{i}(\bar{x})\ge f_{i}(x^{*})\ \mbox{for}\ \ i=1,...,k\ \ \mbox{and}\\
 f_{\ell}(\bar{x})> f_{\ell}(x^{*})\ 
\mbox{for some l}\ \ 1\le \ell\le k.
\end{array}
\end{equation}
\end{Definition}
The problem $(P)$ where all the functions $f_{i}$, $i=1,...,k$ are linear is called a {\em linear multi-objective optimization problem}.



 
 \begin{Remark}
 	\label{remark_vmax}
  The  bi-objective problem 
$$
\begin{array}{ll}
		f_{1}(x)\rightarrow\max,\ f_{2}(x)\rightarrow\min\\
		\mbox{subject to}\\
		x\in\Omega\\
\end{array}
$$
with Pareto solutions $x^{*}\in\Omega$ defined as 
\begin{equation}
\label{eq_max_min}
(f_{1}(\Omega),f_{2}(\Omega))\cap[(f_{1}(x^{*}),f_{2}(x^{*}))+\mathbb{R}_{+-}^{2}]=\{(f_{1}(x^{*}),f_{2}(x^{*}))\}
\end{equation}
where
$$
\mathbb{R}_{+-}^{2}:=\{x=(x_{1},x_{2})\in\mathbb{R}^{2}\ :\ x_{1}\ge 0,\ x_{2}\le 0\}
$$
is equivalent to the problem
$$
\begin{array}{ll}
f_{1}(x)\rightarrow\max,\ -f_{2}(x)\rightarrow\max\\
\mbox{subject to}\\
x\in\Omega\\
\end{array}
$$
in the sense that 
Pareto efficient solution sets (as subsets of the feasible set $\Omega$) coincide and Pareto elements
(the images in $\mathbb{R}^{2}$ of Pareto efficient solutions) differ in sign in the second component.
 \end{Remark}

\subsection{A bi-objective optimization problem related to $(MCKP)$}
\label{Bi-one}

 In relation to the original multiple-choice knapsack problem $(MCKP)$, we consider the linear bi-objective binary optimization problem $(BP1)$  of the form 
$$
\begin{array}{ll}
(BP1)&\begin{array}{l}
\sum_{i=1}^{k}\sum_{j=1}^{n_{i}}p_{ij}x_{ij}\rightarrow\max,\ \sum_{i=1}^{k}\sum_{j=1}^{n_{i}}c_{ij}x_{ij}\rightarrow\min\\
\mbox{subject to}\\
(x_{ij})\in X.
\end{array}
\end{array}
$$
 In this  problem the  left-hand side of the linear inequality constraint 
		$c^{T}x\le b$ of $(MCKP)$ becomes a  second  criterion and the constraint set reduces to the set $X$.
  There are two-fold motivations of considering the bi-objective problem $(BP1)$.

		First motivation comes from the fact
that in $(MCKP)$ the inequality
$$
\sum_{i=1}^{k}\sum_{j=1}^{n_{i}}c_{ij}x_{ij}\le b
$$
is usually seen as a budget (in general: a resource) constraint
with the left-hand-side to be preferably  
not greater than a given available budget $b$. In the bi-objective problem $(BP1)$, this requirement is represented through the minimization of
$\sum_{i=1}^{k}\sum_{j=1}^{n_{i}}c_{ij}x_{ij}$.  In Theorem \ref{theorem1}
	of Section 3, we show that under relatively mild conditions among solutions of the bi-objective problem $(BP1)$ (or the equivalent problem $(BP)$) there are  solutions to problem $(MCKP)$.

 Second motivation is important from the algorithmic point of view 
 and  is related to the fact that in the proposed algorithm we are able to exploit efficiently 
  the  specific structure of the constraint  set $X$ which contains $k$ linear equality  constraints 
  (each one referring to a different group of variables) 
  and the binary conditions only. More precisely, the set $X$ can be represented as the Cartesian
  product
  \begin{equation}
  \label{cartesian_product}
  X=X^{1}\times X^{2}\times .... \times X^{k},
\end{equation}
 of the sets $X^{i}$, where $X^{i}:=\{x^{i}\in\mathbb{R}^{n_{i}}\ |\ \sum_{j=1}^{n_{i}}x_{ij}=1, x_{ij}\in\{0,1\},\ j=1,...,n_{i}\}$, $i=1,...,k$ 
and
	\begin{equation} 
	\label{splitting}
	x=(\underbrace{x_{11},...,x_{1n_{1}}}_{x^{1}},\underbrace{x_{21},...,x_{2n_{2}}}_{x^{2}},...,
	\underbrace{x_{k1},...,x_{kn_{1}}}_{x^{k}})^{T},
	\end{equation}
i.e.,
$$
x=(x^{1},...,x^{k})^{T}
$$
and $x^{i}=(x_{i1},...x_{in_{i}})$. Accordingly,
$$
p=(p^{1},...,p^{k})^{T},\ \ \mbox{and}\ \ c=(c^{1},...,c^{k})^{T}.
$$

	 Note that due to the presence of
the budget inequality constraint  the feasible set $F$ of   problem $(MCKP)$ cannot be represented in the form  analogous to \eqref{cartesian_product}.

According to Remark \ref{remark_vmax}, problem $(BP1)$ can be 
equivalently reformulated in the form
$$
\begin{array}{ll}
(BP)&\begin{array}{l}
\text{V}\max\ (p^{T}x,\ (-c)^{T}x)\\
\mbox{subject to}\\
x\in X.
\end{array}
\end{array}
$$

\section{The relationships between $(BP)$ and $(MCPK)$}

Starting from the multiple-choice knapsack problem $(MCKP)$ of the form
$$
\begin{array}{ll}
(MCKP)&\begin{array}{l}
\max \ p^{T}x\\
\text{subject to}\\
x\in F,\\
\end{array}
\end{array}
$$
in the present section we analyse relationships between  problems 
 $(MCKP)$ and $(BP)$. 
  
We start with a basic observation. Recall first that $(MCKP)$ is solvable, i.e., the feasible  set $F$ is nonempty if
$$
b\ge \min_{x\in X} c^{T}x.
$$
On the other hand, if $b\ge\max_{x\in X} c^{T}x$,  
$(MCKP)$ is trivially solvable. Thus, in the sequel we assume that
\begin{equation}
\label{eq_lim}
C_{min}:=\min_{x\in X} c^{T}x\le b<C_{max}:=\max_{x\in X} c^{T}x.
\end{equation}

Let  
$P_{max}:=\max_{x\in X} p^{T}x$, i.e., $P_{max}$ is the maximal value
	of the function $p^{T}x$ on the set $X$.}
The following observations are essential  for further considerations.
	\begin{enumerate}
	\item First,  among the elements of $X$ which realize the maximal value $P_{max}$, there exists at least one which is feasible for $(MCKP)$, i.e., there exists $x_{p}\in X$,
	$p^{T}x_{p}= P_{max}$ such that $c^{T}x_{p}\le b$, i.e.,
	\begin{equation}
	\label{case1}
	C_{min}\le c^{T}x_{p}\le b<C_{max}.
\end{equation}
Then, clearly,  $x_{p}$ solves
$(MCKP)$.
\item  Second,   none  of elements 
	 which realize the maximal value $P_{max}$ is feasible for $(MCKP)$, i.e., for every $x_{p}\in X$,
	$p^{T}x_{p}= P_{max}$ we have  $c^{T}x_{p}> b$, i.e.,
	any $x_{p}$  realizing the maximal value $P_{max}$ is infeasible for $(MCKP)$,  i.e.
	\begin{equation}
	\label{case2}
	C_{min}\le b<c^{T}x_{p}\le C_{max}.  
\end{equation}
\end{enumerate}

In the sequel, we concentrate on  Case 2, characterized by  \eqref{case2}. This case is related to problem $(BP)$.  To see this let us  introduce some additional notations. Let $x_{cmin}\in X$ and $x_{pmax}\in X$ be defined as
$$
\begin{array}{lll}
c^{T}x_{cmin}=C_{min}\ &\text{and\ }\ & p^{T}x_{cmin}=\max_{c^{T}x=C_{min}}p^{T}x\\
p^{T}x_{pmax}=P_{max}\ &\text{and\ }\  & c^{T}x_{pmax}=\min_{p^{T}x=P_{max}}c^{T}x.
\end{array}
$$
 Let $S_{bo}$  be the  set of  all Pareto solutions to the bi-objective problem $(BP)$, 
	$$
	S_{bo}:=\{x\in X\ :\ (p^{T},(-c)^{T})(X)\cap [(p^{T}x,(-c)^{T}x)+\mathbb{R}^{2}_{+}]=\{(p^{T}x,(-c)^{T}x)\}\},
	$$
(c.f.	Definition \ref{pareto_sol}).
	The following lemma holds.
	\begin{Lemma} 
		\label{Lemma_1}
		Assume that we are in Case 2, i.e., condition  $(7)$ holds. There exists a Pareto solution to the bi-objective optimization problem $(BP)$, $\bar{x}\in S_{bo}$ which is feasible to problem $(MCKP)$, i.e. $c^{T}\bar{x}\le b$ which amounts to $\bar{x}\in F$.
		\end{Lemma}
		
		\begin{proof} 
According to Definition \ref{pareto_sol}, both $x_{pmax}\in X$ and $x_{cmin}\in X$ are Pareto efficient solutions to $(BP)$, i.e., there is no $x\in X$ such that
$(p^{T}x, c^{T}x)\neq (p^{T}x_{pmax}, c^{T}x_{pmax})$  and
$$
p^{T}x\ge p^{T}x_{pmax}\   \mbox{and} \ c^{T}x\le c^{T}x_{pmax}
$$
and there is no $x\in X$ such that
$(p^{T}x, c^{T}x)\neq (p^{T}x_{cmin}, c^{T}x_{cmin})$ and
$$
p^{T}x\ge p^{T}x_{cmin}\   \mbox{and} \ c^{T}x\le c^{T}x_{cmin}.
$$
Moreover, by \eqref{case2}, 
\begin{equation}
\label{case3}
C_{min}=c^{T}x_{cmin}\le b<c^{T}x_{pmax}.
\end{equation}
In view of \eqref{case3},  $\bar{x}=x_{cmin}\in S_{bo}$ and $\bar{x}=x_{cmin}\in F$  ($c^{T}x_{cmin}\le b$)
which means that $\bar{x}$ is feasible to problem $(MCKP)$, which concludes the proof. 
\end{proof}

Now we are ready to formulate  the result  establishing the relationship between solutions of $(MCKP)$ and Pareto efficient solutions of $(BP)$ in the  case where the condition \eqref{case2} holds.

\begin{Theorem}
	\label{theorem1}
	Suppose we are given problem $(MCKP)$ satisfying condition
	\eqref{case2}. Let $x^{*}\in X$ be a Pareto solution to $(BP)$,
	such that 
	$$
	b-c^{T}x^{*}=\min_{x\in S_{bo},b-c^{T}x\ge 0} b-c^{T}x.\ \ \ \ \ \ (*)
	$$
Then $x^{*}$ solves $(MCKP)$.
\end{Theorem}
\begin{proof}
	Observe first that, by Lemma \ref{Lemma_1}, there exist
	$x\in S_{bo}$ satisfying the constraint $c^{T}x\le b$, i.e., condition $(*)$ is not dummy.
	
	By contradiction, suppose that a feasible  element $x^{*}\in F$, i.e., $x^{*}\in X$, $c^{T}x^{*}\le b$,  is 
	not a solution to $(MCKP)$, i.e., there exists an $x_{1}\in X$,
	such that 
	$$
	c^{T}x_{1}\le b\ \ \mbox{and}\ p^{T}x_{1}>p^{T}x^{*}.\ \ \ (**)
	$$
	We show that $x^{*}$ cannot satisfy condition $(*)$.
	If $c^{T}x_{1}\le c^{T} x^{*}$, then  $x^{*}$ is not a
	Pareto solution to $(BP)$, i.e., $x^{*}\not\in S_{bo}$, and $x^{*}$ does not satisfy condition $(*)$.
	Otherwise,  $c^{T}x_{1}> c^{T} x^{*}$, i.e., 
	\begin{equation} 
	\label{eq_main}
	b-c^{T}x^{*}> b-c^{T} x_{1}.
	\end{equation}
	
If $x_{1}\in S_{bo}$, then because
	  $x^{*}\in S_{bo}$, then, according to \eqref{eq_main}, $x^{*}$ cannot satisfy
	 condition $(*)$.	

	If  
	$x_{1}\not\in S_{bo}$, there exists $x_{2}\in S_{bo}$ which dominates
	$x_{1}$, i.e., $(p^{T}x_{2}, (-c)^{T}x_{2})\in (p^{T}x_{1}, (-c)^{T}x_{1})+\mathbb{R}_{+}^{2}$.  Again, if 
	 $c^{T}x_{2}\le c^{T} x^{*}$, then  $x^{*}$ is not a
	Pareto solution to $(BP)$, i.e. $x^{*}$ cannot satisfy condition $(*)$.
	 Otherwise, if  $c^{T}x_{2}> c^{T} x^{*}$, then either $x^{*}\not\in S_{bo}$ and consequently
	 $x^{*}$ cannot satisfy condition $(*)$, or $x^{*}\in S_{bo}$, in which case
	 $b-c^{T}x^{*}> b-c^{T} x_{1}$ and $x^{*}$
	does not satisfy condition $(*)$,  a contradiction which completes the proof.
\end{proof}

Theorem \ref{theorem1} says that under condition \eqref{case2}
any solution to $(BP)$ satisfying condition $(*)$ solves 
problem $(MCKP)$. General relations between constrained optimization and multi-objective programming were investigated in \cite{Klamroth2007}.

  Basing ourselves on Theorem \ref{theorem1}, in Section 5 we provide an algorithm
	for finding $x\in S_{bo}$, a Pareto solution to $(BP)$, which is feasible to 
	problem $(MCKP)$ and for which the condition $(*)$  is either satisfied or is, in some sense, as close as possible to be satisfied. In this latter case, the algorithm provides  upper and lower bounds for the optimal value of $(MCKP)$.

\begin{figure}
	\includegraphics[width=0.79\textwidth]{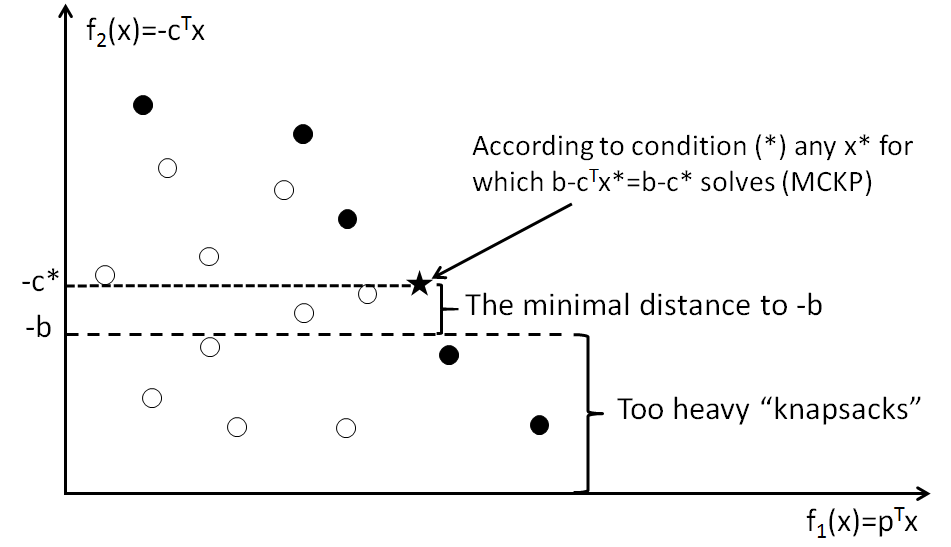} 
	\caption{Illustration to the content of Theorem 3.1; black dots - outcomes of Pareto efficient solutions to (BP), star - Pareto efficient outcome to (BP) which solves (MCKP).}
	\label{rys1}
\end{figure}

\section{Decomposition of the scalarized bi-objective  problem $(BP)$}

In the present section, we consider  problem $(BS(\lambda_{1},\lambda_{2}))$
	 defined by \eqref{eq-scalar2} which is a linear scalarization of problem $(BP)$. In our algorithm BISSA, presented in Section 5, we obtain an approximate feasible solution to $(MCKP)$ by solving a (usually very small) number of problems of the form $(BS(\lambda_{1},\lambda_{2}))$. The main advantage of basing our algorithm  on problems  $(BS(\lambda_{1},\lambda_{2}))$ is that they are explicitly
	solvable by simple closed-form expressions \eqref{eq-2s} .

For problem $(BP)$  the following classical scalarization result holds.

\begin{Theorem}\cite{Ehrgott2005,Miettinen1999}
\label{miettinen}
If there exist $\lambda_{\ell}>0$, $\ell=1,2$, such that
$x^{*}\in X$ is a solution to the scalarized problem
\begin{equation}
\label{eq-scalar2}
\begin{array}{ll}
(BS(\lambda_{1},\lambda_{2}))&\begin{array}{l}
\max\limits_{x\in X} \  \lambda_{1} p^{T}x+\lambda_{2} (-c)^{T}x\\
\end{array}
\end{array}
\end{equation}
then $x^{*}$ is a Pareto efficient solution to  problem $(BP)$.
\end{Theorem}

Without loosing generality we can assume 
that $\sum_{l=1}^{2}\lambda_{\ell}=1$. In the sequel, we consider, for $0<\lambda<1$, scalarized problems of the form
\begin{equation}
\label{eq-scalar}
\begin{array}{ll}
(BS(\lambda))&\begin{array}{l}
\max\limits_{x\in X} \lambda p^{T}x+(1-\lambda) (-c)^{T}x\\
\end{array}
\end{array}
\end{equation}

\begin{Remark}
	\label{remark_scalar}
	According to Theorem \ref{miettinen}, solutions to problems
	\begin{equation}
	\label{pomocnicze}
	\max\limits_{x\in X} p^{T}x,\ \ \	\max\limits_{x\in X} (-c)^{T}x
	\end{equation}
	need not be Pareto efficient because the weights are not both positive. However, there exist Pareto efficient solutions 
	to $(BP)$ among solutions to these problems. 
	
	Namely, there exist $\varepsilon_{1}>0$ and 
	$\varepsilon_{2}>0$ such that solutions to problems
	$$
	\begin{array}{ll}
		(P1)&\max\limits_{x\in X} p^{T}x+\varepsilon_{1}(-c)^{T}x
		\end{array}
		$$
			and
			$$
		\begin{array}{ll}
			(P2)&\max\limits_{x\in X} (-c)^{T}x+\varepsilon_{2}p^{T}x
			\end{array}
			$$
	are Pareto efficient solutions to problems \eqref{pomocnicze}, respectively. Suitable $\varepsilon_{1}$ and $\varepsilon_{2}$ will be determined in the next subsection.
\end{Remark}

\subsection{Decomposition}

Due to the highly structured form of the set $X$ and the possibility of representing $X$ in the form \eqref{cartesian_product},
$$
X=X^{1}\times X^{2}\times .... \times X^{k},
$$
we can provide explicit formulae for solving problems $(BS(\lambda))$. To this aim we decompose problems $(BS(\lambda))$ as follows.

Recall that by using the notation \eqref{splitting} we can put any $x\in X$ in the form
$$
x:=(x^{1},x^{2},...,x^{k})^{T},
$$
where  $x^{i}=(x_{i1},....,x_{in_{i}})$, $i=1,...,k$, and $\sum_{j=1}^{n_{i}}x_{ij}=1$.

Let $0<\lambda<1$.  According to \eqref{cartesian_product} we have
$$
X^{i}:=\{x^{i}=(x_{i1},....,x_{in_{i}})\in\mathbb{R}^{n_{i}}\ :\ 
\sum_{j=1}^{n_{i}}x_{ij}=1,\ \ \
x_{ij}\in\{0,1\},\ \ j=1,...,n_{i}\}
$$
for $i=1,...,k$. 
Consider problems $(BS(\lambda))_{i}$, $i=1,...,k$,  of the form
\begin{equation} 
\label{decom-problem}
\begin{array}{ll}
(BS(\lambda))_{i}&\begin{array}{l} 
\max\limits_{x^{i}\in X^{i}} [\lambda (p^{i})^{T}x^{i}+(1-\lambda) (-c^{i})^{T}x^{i}]\\
\end{array}
\end{array}
\end{equation}
By solving problems $(BS(\lambda))_{i}$, $i=1,...,k$, we find their solutions $\bar{x}^{i}$. We shall show that
$$
\bar{x}:=(\bar{x}^{1},...,\bar{x}^{k})^{T}
$$
solves $(BS(\lambda))$. Thus,  problem \eqref{eq-scalar}  is decomposed into $k$ subproblems \eqref{decom-problem},  the solutions of which form solutions to \eqref{eq-scalar}.

Note that similar decomposed problems with feasible sets $X^i$ and another objective functions have already been considered in \cite{Cherfi2010} in relation to  multi-dimensional multiple-choice knapsack problems.

Now we give a closed-form formulae for solutions of $(BS(\lambda))_{i}$. For $i=1,..,,k$, let 
\begin{equation}
\label{eq-1v}
V_{i}:=\max\{\lambda p_{ij}+(1-\lambda) (-c_{ij}) \ :\ 1\le j\le n_{i}\}.
\end{equation}
and let  $1\le j^{*}_{i}\le n_{i}$ be the index number for which the value $V_{i}$ is attained, i.e.,
\begin{equation}
\label{eq-2v}
V_{i}=\lambda p_{ij_{i}^{*}}+(1-\lambda) (-c_{ij_{i}^{*}}).
\end{equation}
We show that
\begin{equation}
\label{eq-1s}
\bar{x}^{i}:=(0,..,\underbrace{1}_{j_{i}^{*}},0,...,0)\ \ \ 
\end{equation}
is a solution to $(BS(\lambda))_{i}$ 
 and
\begin{equation}
\label{eq-2s}
\bar{x}^{*}:=(\bar{x}^{1},\bar{x}^{2},...,\bar{x}^{k})^{T}
\end{equation}
is a solution to $(BS(\lambda))$. The  optimal value of $(BS(\lambda))$ is
\begin{equation}
\label{eq-2sa}
V:=V_{1}+...+V_{k}.
\end{equation}

Namely, the following proposition holds.

\begin{Proposition}
\label{prop-2}
Any element $\bar{x}^{i}\in\mathbb{R}^{n_{i}}$ given by \eqref{eq-1s} solves $(BS(\lambda))_{i}$ for $i=1,...,k$
and any $\bar{x}^{*}\in\mathbb{R}^{n}$
given by \eqref{eq-2s} solves problem $(BS(\lambda))$.
\end{Proposition}
\begin{proof}
Clearly, $\bar{x}^{i}$ are feasible for $(BS(\lambda))_{i}$, $i=1,...,k$,
 because $\bar{x}^{i}$ is of the form \eqref{eq-1s} and hence 
	belongs to the set $X^{i}$ which is the constraint set of  $(BS(\lambda))_{i}$.
Consequently, $\bar{x}^{*}$ defined by \eqref{eq-2s} is feasible for $(BS(\lambda))$
	because all the components are binary and the linear equality constraints
	$$
	\sum_{j=1}^{n_{i}} x^{i}_{j}=1, \ \ i=1,2,...,k
	$$
	are satisfied.

To see that $\bar{x}^{i}$ are also optimal for $(BS(\lambda))_{i}$, $i=1,...,k$, suppose by the contrary, that
there exists $1\le i\le k$ and an element $y\in\mathbb{R}^{n_{i}}$ which is feasible for $(BS(\lambda))_{i}$
with the value of the objective function strictly greater than the value at $\bar{x}^{i}$, i.e.,
$$
\sum_{j=1}^{n_{i}}[\lambda p_{ij}+(1-\lambda) (-c_{ij})]y_{j}>
\sum_{j=1}^{n_{i}}[\lambda p_{ij}+(1-\lambda) (-c_{ij})]\bar{x}^{i}_{j}.
$$
This, however, would mean that there exists an index $1\le j\le n_{i}$
such that
$$
\lambda p_{ij}+(1-\lambda) (-c_{ij})>\lambda p_{ij^{*}}+(1-\lambda) (-c_{ij^{*}})
$$
contrary to the definition of $j^{*}$.

To see that $\bar{x}^{*}$ is  optimal for $(BS(\lambda))$,  suppose by the contrary, that
there exists  an element $y\in\mathbb{R}^{n}$ which is feasible for $(BS(\lambda))$
and the value of the objective function at $y$ is strictly greater than the value of the objective function at $\bar{x}^{*}$, i.e.,
$$
\lambda p^{T}y+(1-\lambda) (-c)^{T}y>\lambda p^{T}\bar{x}^{*}+(1-\lambda) (-c)^{T}\bar{x}^{*}.
$$
In the same way as previously, we get the contradiction with the definition of the components of $\bar{x}^{*}$ given by \eqref{eq-2s}.
\end{proof}

Let us observe that each optimization problem $(BS(\lambda))_{i}$ can be solved in time $O(n_{i})$, hence problem $(BS(\lambda))$ can be solved in time $O(n)$, where $n=\sum_{i=1}^{k}n_{i}$.

Clearly, one can have more than one solution to $(BS(\lambda))_{i}$, $i=1,...,k$. In the next section, according to Theorem \ref{theorem1}, from among all the solutions of
$(BS(\lambda))$ we choose the one for which the value of the second criterion is greater than  and as close as possible to $-b$.

Note that by using Proposition \ref{prop-2}, one can easily solve problems $(P1)$ and $(P2)$ defined in Remark \ref{remark_scalar},
i.e., by applying  \eqref{eq-2sa} we immediately get
$$
F_{1}:=\max_{x\in X} p^{T}x,\ \ \ F_{2}:=\max_{x\in X} (-c)^{T}x
$$
the optimal values of $(P1)$ and  $(P2)$ and by  \eqref{eq-2s}, we find their solutions $\bar{x}_{1}$ and $\bar{x}_{2}$, respectively.
  
Proposition \ref{prop-2} and formula \eqref{eq-2s} allows  to find
$\varepsilon_{1}>0$ and $\varepsilon_{2}>0$ 
as defined in Remark \ref{remark_scalar}. 
By \eqref{eq-2s}, it is easy to find  elements $\bar{x}_{1},\bar{x}_{2}\in X$ such that
$$
F_{1}= p^{T}\bar{x}_{1},\ \ F_{2}= (-c)^{T}\bar{x}_{2}.
$$
Put
$$
\bar{F}_{1}:= p^{T}\bar{x}_{2},\ \ \bar{F}_{2}:= (-c)^{T}\bar{x}_{1}
$$
and let 
$$
\bar{V}_{1}:=F_{1}-decr(p),\ \ \bar{V}_{2}:=F_{2}-decr(-c),
$$
where $decr(p)$ and $decr(-c)$ denote the smallest nonzero decrease on $X$  of functions $p$ and $(-c)$ from ${F}_{1}$ and ${F}_{2}$, respectively. Note that
$decr(p)$ and $decr(-c)$ can easily  be  found.  

\begin{Remark}
	\label{remark_decr}
	The following formulas describe $decr(p)$ and $decr(-c)$, 
	\begin{equation}
	\label{decr}
	\ decr(p):=\min_{1\le i \le k} (p_{max}^{i}-p_{submax}^{i}),\ \ decr(-c):=\min_{1\le i \le k} ((-c)_{max}^{i}-(-c)_{submax}^{i}),
	\end{equation}
	where  $p^{i}$ and $c^{i}$, $i=1,...,k$, are defined by \eqref{splitting}, $p_{submax}^{i}$, $(-c)_{submax}^{i}$, $i=1,...,k$, are submaximal values of functions $(p^{i})^{T}x^{i}$,
	$((-c)^{i})^{T}x^{i}$, $x^{i}\in X^{i}$, $i=1,...,k$. 
	
	For any $1\le i\le k$, the submaximal values of a linear function $(d^{i})^{T}x^{i}$ on $X^{i}$ can be found
	by:  ordering first the coefficients of the function $(d^{i})$ decreasingly,
	$$
	(d^{i})^{j1}>(d^{i})^{j2}\ge... \ge (d^{i})^{jm_{i}},
	$$
	and next observing that the submaximal (i.e., smaller than maximal but as close as possible to the maximal) value of $(d^{i})$ on $X^{i}$ is attained for 
	$$
	\bar{x}^{i}:=(0,...,\underbrace{1}_{j2},0,...0).
	$$
\end{Remark}

Basing on Remark \ref{remark_decr} one can find values of $p_{submax}^{i}$ and $(-c)_{submax}^{i}$ in time $O(n_{i})$, $i=1,...,k$, even without any sorting. It can be done for a given $i$ by finding a maximal value among all $p_{ij}$ ($c_{ij}$), $j=1,...,n_{i}$,  except  $p_{max}^{i}$ ($c_{max}^{i}$). Therefore the computational cost of calculating $decr(p)$ and $decr(-c)$ is $O(n)$.

We have the following fact.

\begin{Proposition}
	\label{prop3}
	 Let $F_{1}$, $F_{2 }$, $\bar{F}_{1}$, $\bar{F}_{2}$,  $\bar{V}_{1}$, $\bar{V}_{2}$ be as defined above. The problems
		$$
		\begin{array}{ll}
		(P1)&\max\limits_{x\in X} p^{T}x+\varepsilon_{1}(-c)^{T}x
		\end{array}
		$$
		and
		$$
		\begin{array}{ll}
		(P2)&\max\limits_{x\in X} (-c)^{T}x+\varepsilon_{2}p^{T}x
		\end{array}
		$$
	where
	\begin{equation}
	\label{epsilon}
	\varepsilon_{1}:=\frac{F_{1}-\bar{V}_{1}}{F_{2}-\bar{F}_{2}},\ \ \  
	\varepsilon_{2}:=\frac{F_{2}-\bar{V}_{2}}{F_{1}-\bar{F}_{1}},
	\end{equation}
	give Pareto efficient solutions to problem $(BP)$, $\bar{x}_{1}$ and $\bar{x}_{2}$, respectively. Moreover,
	$$
	f_{1}(\bar{x}_{1})=F_{1}\ \ \text{and}\ \ 	f_{2}(\bar{x}_{2})=F_{2},
	$$
	i.e., $\bar{x}_{1}$ $\bar{x}_{2}$ solve problems \eqref{pomocnicze}, respectively.
\end{Proposition}	
\begin{proof} Follows immediately from the adopted notations, see Fig. \ref{rys2}.
	For instance, the objective  of problem (P1) is represented by the  straight line passing through
	points $(F_{1},\bar{F}_{2})$ and $(\bar{V}_{1},F_{2})$, i,e,
	$$
	F_{1}+\varepsilon_{1}\bar{F}_{2}=\bar{V}_{1}+\varepsilon_{1}F_{2}
	$$
	which gives \eqref{epsilon}. The  choice  of $F_{1}$, $\bar{F}_{2}$ and $\bar{V}_{1}$,  $F_{2}$ guarantees that $\bar{x}_{1}$ solves $(P1)$ (and  analogously for $\bar{x}_{2}$ which solves $(P2)$).
	\end{proof}

\begin{figure}
	\includegraphics[width=0.79\textwidth]{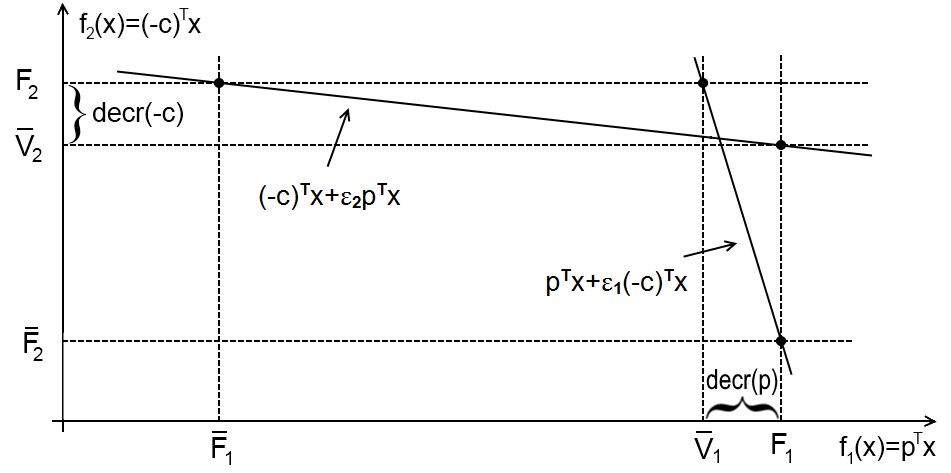} 
	\caption{Construction of $\varepsilon_{1}$ and $\varepsilon_{2}$.}
	\label{rys2}
\end{figure}

\section{Bi-objective Approximate Solution Search Algorithm $(BISSA)$ for solving $(MCKP)$}

In this section, we propose the bi-objective approximate solution search algorithm $BISSA$, for finding an  element $\hat{x}\in F$ which is an approximate solution to  $(MCKP)$. The algorithm relies on solving a series of problems $(BS(\lambda))$ defined by
	\eqref{eq-scalar} for   $0<\lambda<1$  chosen in the way that the Pareto solutions $x(\lambda)$ to $(BS(\lambda))$ are feasible for $(MCKP)$ and  for which
		$(-c)^{T}x(\lambda)+b\ge 0$ and $(-c)^{T}x(\lambda)+b$
		diminishes for subsequent $\lambda$.

According to Theorem \ref{miettinen}, each solution to
$(BS(\lambda))$ solves the linear bi-objective optimization problem $(BP)$, 
$$
\begin{array}{ll}
(BP)&\begin{array}{l}
\text{Vmax} (p^{T}x,(-c)^{T}x)\\
\text{subject to\ \ }
x\in X.
\end{array}
\end{array}
$$
According to Theorem \ref{theorem1}, any Pareto efficient solution $x^{*}$ to 
problem $(BP)$  which is feasible to $(MCKP)$, i.e.,  $(-c)^{T}x^{*}\ge -b$,  and satisfies condition $(*)$, i.e.,
$$
(-c)^{T}x^{*}+b=\min_{x\in S_{bo},
(-c)^{T}x+b\ge 0} (-c)^{T}x+b\ \ \ \ \ \ (*)
$$
solves problem $(MCKP)$. Since  problems $(BS(\lambda))$ are defined with the help of linear scalarization, we are   not able, in general, to enumerate all $x\in S_{bo}$ such that $(-c)^{T}x+b\ge 0$  in order to find an $x^{*}$ which satisfy condition  $(*)$. On the other hand,
	by using linear scalarization, we are able to decompose and easily solve problems $(BS(\lambda))$.

The $BISSA$ algorithm aims at finding
a Pareto efficient solution $\hat{x}\in X$ to $(BP)$ 
which is feasible to $(MCKP)$, i.e., $c^{T}\hat{x}\le b$
for which  the value of 
$b-c^{T}\hat{x}$ 
is as small as possible (but not necessarily minimal) and approaches  
condition $(*)$ of Theorem \ref{theorem1} as close as possible.
 
 Here, we give a description of the $BISSA$ algorithm. The first step of the algorithm (lines 1-5) is to find solutions to problems $(P1)$ and $(P2)$ as well as their outcomes. The solutions are the extreme Pareto solutions to problem $(BP)$. Those points named ${(a_{1},b_{1})}^0$ and ${(a_{2},b_{2})}^0$  are presented in Fig. \ref{rys3}.
Then (lines 6-9), in order to assert whether a solution to problem ($MCKP$) exists or not, a basic checking is made against value $-b$.
If the algorithm reaches line 10, no solution has been found yet, and we can begin the exploration of the search space. 

We calculate $\lambda$ according to line 13. The value of $\lambda$ is the slope of the  straight line joining $(a_{1},b_{1})$ and $(a_{2},b_{2})$. At the same time it is
 the scalarization parameter defining the problem $(BS(\lambda))$ (formula \eqref{eq-scalar}). 
The outcome of the solution to problem $(BS(\lambda))$ cannot lie below the straight line determined by points $(a_{1},b_{1})$ and $(a_{2},b_{2})$.
It must lie on or above this line, as it is the Pareto efficient solution to problem $(BP)$.
Then, problem $(BS(\lambda))$ is solved (line 14) by using formulae \eqref {eq-1s} and \eqref {eq-2s}.
 Next, in lines 15-27 of the \textit{repeat-until} loop a  scanning of the search space
is conducted to find solutions to problem $(BP)$ which are feasible to 
problem $(MCKP)$. 
If there exist solutions with outcomes
lying above the straight line determined by $\lambda$ (the condition in line 15 is true), either the narrowing of the search space is made (by determining
new points $(a_{1},b_{1})$ and $(a_{2},b_{2})$, see Fig.\ref{rys3}, and points with upper index equal to $1$), and the loop continues, or the solution to problem ($MCKP$) is derived. If not, the solution
$x$ from set $S$ which outcome lies above the line determined by $-b$ (the feasible solution to problem ($MCKP$)) and for which value $f_{2}(x)+b$ is
minimal in this set, is an approximate solution ($\hat{x}$) to problem ($MCKP$), and the loop terminates. Finally (line 28), the upper
bound $f_{1}(\hat{x})+u$ on the profit value of exact solution to problem ($MCKP$) is calculated.



	\begin{algorithm}
		\caption{$BISSA$ - Approximate solution search to $(MCKP)$}\label{scan_u_d}
		\begin{algorithmic}[1]
			
			\State{Calculate $\varepsilon_{1}$,
				$\varepsilon_{2}$ according to \eqref{epsilon} } 
			\State{Assume that $f_{1}(x)=p^{T}x$ and $f_{2}(x)=(-c)^{T}x$}
			\State{Solve $(P1)$  according to \eqref{eq-2sa} and \eqref{eq-2s}}
			\Comment{ $x_{1}$ a solution to $(P1)$}
				\State{Solve  $(P2)$ according to  \eqref{eq-2s} and \eqref{eq-2sa}}\Comment{  $x_{2}$ a solution to $(P2)$}
			\State{$a_{1}:=f_{1}(x_{1})$, $b_{1}:=f_{2}(x_{1})$, $a_{2}:=f_{1}(x_{2})$, $b_{2}:=f_{2}(x_{2})$}
			\State{\text{\bf if} $(a_{1}, b_{1})=(a_{2}, b_{2})$ and  $ b_{2}\ge -b$ \text{\bf then} $x_{2}$ solves  $(MCKP)$ and STOP \text{\bf end if}}


			\State{\text{\bf if} $b_{1}\ge -b$ \text{\bf then }$x_{1}$ solves ($MCKP$) and STOP \text{\bf end if}}

			\State{\text{\bf if} $b_{2}=-b$ \text{\bf then} $x_{2}$ solves ($MCKP$) and STOP \text{\bf end if}}

			 \State{\text{\bf if} $b_{2}<-b$ \text{\bf then} no solution to ($MCKP$) and STOP \text{\bf end if}}
			\State{} 
			\Comment{$(a_{1}, b_{1})\ne(a_{2}, b_{2})\text{ and }b_{1}<-b<b_{2}$. Explore the search space}
			
			\State{$loop:=TRUE$}

			\Repeat
				\State{$\lambda:=\frac{(b_{2}-b_{1})}{(a_{1}-a_{2})+(b_{2}-b_{1})}$, $\alpha:=\lambda a_{1}+(1-\lambda)b_{1}$} \Comment{$0<\lambda<1$}
				\State{Solve $(BS(\lambda))$ according to \eqref{eq-1s} and   \eqref{eq-2s}}
				\Comment{$x$ a solution, $opt$ the optimal value, $S$ the solution set to $(BS(\lambda))$}
			
				\If{$opt>\alpha$}
					\If{$f_{2}(x)> -b$}
						\State $a_{2}:=f_{1}(x)$, $b_{2}:=f_{2}(x)$
					\ElsIf{$f_{2}(x)< -b$}
						\State $a_{1}:=f_{1}(x)$, $b_{1}:=f_{2}(x)$
					\Else
						\State $x$ solves ($MCKP$) and STOP
					\EndIf
				\Else \Comment $opt=\alpha$
					\State  $\hat{x}:=\text{arg}\min_{x \in S,  f_{2}(x)\ge -b}{f_{2}(x)}$
					\State $loop:=FALSE$
				\EndIf

								\Until{$\neg loop$} \Comment{ $\hat{x}$ is an approximate solution to  ($MCKP$) }
								\Comment{$f_{1}(\hat{x})$ is a lower bound ($LB$) for ($MCKP$)}
								\State{$u:=\frac{(a_{1}-f_{1}(\hat{x}))(f_{2}(\hat{x})+b)}{f_{2}(\hat{x})-b_{1}}$}
								\Comment{$f_{1}(\hat{x})+u$ is an upper bound ($UB$) for ($MCKP$)}
		\end{algorithmic}
	\end{algorithm}

The $BISSA$ algorithm finds either an exact solution to problem ($MCKP$), or (after reaching line 27) a lower bound (LB) with its solution $\hat{x}$ and an upper bound (UB) (see Fig.\ref{rys3}).
A solution found by the algorithm is, in general, only an approximate solution to problem ($MCKP$) because a triangle (called further the triangle of uncertainty) determined by
points $(f_{1}(\hat{x}),  f_{2}(\hat{x})), (f_{1}(\hat{x})+u, -b), (f_{1}(\hat{x}),-b)$ may contain other Pareto outcomes (candidates for outcomes of exact solutions to problem  ($MCKP$))
which the proposed algorithm is not able to derive. The reason is that we use a  scalarization technique based on weighted sums of criteria functions to obtain Pareto solutions to problem $(BP)$.
\begin{figure}
\includegraphics[width=0.79\textwidth]{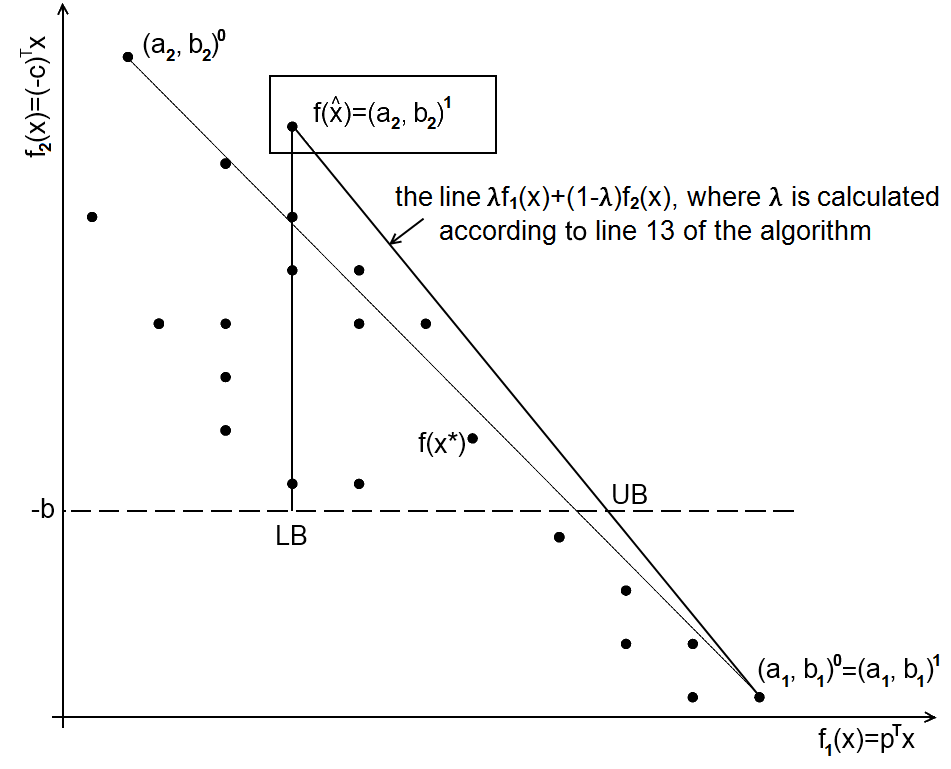} 
\caption{Outcome $f(\hat{x})$ and bounds derived by the $BISSA$ algorithm; $x^{*}$ -- the solution to problem $(MCKP)$.}
\label{rys3}
\end{figure}

Let us recall that each instance of the optimization problem $(BS(\lambda))$ can be solved in time $O(n)$, but
the number of these instances solved by the proposed algorithm depends on the size of the problem (values $k$ and $n_{i}$) and the data.

\section{Computational experiments}
 Most publicly available test instances refer  not to 
	the $(MCKP)$ problem (let us recall, that there is only one inequality or budget constraint in the problem we consider) but to  multi-dimensional knapsack problems. Due to this fact we generate new random instances (available from the authors on request). However, to compare solutions obtained by the $BISSA$ algorithm to the exact solutions we used the minimal algorithm for the multiple-choice knapsack problem \cite{Pisinger1995} which we call $EXACT$ and its implementation in C \cite{Pisinger1995code}. The $EXACT$ algorithm gives the profit value of the optimal 
solution as well as the solution obtained by the greedy algorithm for the $(MCKP)$ problem, so the quality of the $BISSA$ algorithm approximate solutions can be assessed  in terms of the difference or relative difference between profit values of approximate solutions and exact ones.

Since the difficulty of knapsack problems (see, e.g., the monograph \cite{Martello1990}) depends on the correlation between profits and weights of items, we conducted two computational experiments: Experiment 1 with uncorrelated data instances (easy to solve) and Experiment 2 with weakly correlated data instances (more difficult to solve) (c.f.\cite{Han2010}). The explanation why weakly correlated problems are more difficult to solve by the 
 $BISSA$ algorithm than uncorrelated ones we give later.

To prepare test problems (data instances) we used a method proposed in \cite{Pisinger1995} and our own procedure for calculating total cost values.\\
The $BISSA$ algorithm has been implemented in C.
The implementation of $BISSA$ algorithm was run on off-the-shelf laptop (2GHz AMD processor, Windows 10), and
the implementation of $EXACT$ algorithm was run on PC machine (4x3.2GHz Intel processor, Linux). The running time
for $BISSA$ and $EXACT$ algorithms for each of the test problems was below one second.

The contents of the tables columns containing experiment results is as follows.
\begin{description}
\item 1 -- problem no.
\item 2 -- profit of the exact solution found by the $EXACT$ algorithm.
\item 3 -- profit of the approximate solution found by the $BISSA$ algorithm.
\item 4 -- difference between $2$ and $3$.
\item 5 -- relative (\%) difference between $2$ and $3$.
\item 6 -- upper bound for $(MCKP)$ found by the $BISSA$ algorithm.
\item 7 -- the difference between the upper bound and profit of the approximate solution.
\item 8 -- the relative difference between the upper bound and profit of the approximate solution.
\item 9 -- upper bound for $(MCKP)$ found by the greedy algorithm.
\item 10 -- number of $(BS(\lambda))$ problems solved by the $BISSA$ algorithm.\\
\end{description}

{\bf Experiment 1 -- uncorrelated data (unc) instances}\\
We generated 10 test problems assuming that $k=10$ and $n_{i}=1000, i=1,...,k$ (problem set $(unc, 10, 1000)$),
10  test problems assuming that $k=100$ and $n_{i}=100, i=1,...,k$ (problem set $(unc, 100, 100)$), and
10  test problems assuming that $k=1000$ and $n_{i}=10, i=1,...,k$ (problem set $(unc, 1000, 10)$).
For each test problem profits ($p_{ij}$) and costs ($c_{ij}$) of items
were randomly distributed (according to the uniform distribution) in $[1, R]$, $R=10000$. Profits and costs of items were integers. For each test problem the total cost $b$ was equal to either $c + random(0,\frac{1}{4}*c)$ or $c - random(0,\frac{1}{4}*c)$ randomly selected with the same probability equal to $0.5$), where $c=\frac{1}{2}\sum_{i=1}^{k}(\min_{j=1,...,n_{i}}c_{ij}+\max_{j=1,...,n_{i}}c_{ij})$, and
$ random(0,r)$ denotes randomly selected (according to the uniform distribution) integer from $[0,r]$.
\\The results for problem sets $(unc, 10, 1000)$, $(unc, 100, 100)$ and $(unc, 1000, 10)$ are given, respectively, in tables Table 1, Table 2 and Table 3.\\

{\bf Experiment 2 -- weakly correlated (wco) data instances}\\
We generated 10 test problems assuming that $k=20$ and $n_{i}=20, i=1,...,k$ (problem set $(wco, 20, 20)$). For each test problem costs ($c_{ij}$) of items in set $N_{i}$ 
were randomly distributed (according to the uniform distribution) in $[1, R]$, $R=10000$, and profits of items ($p_{ij}$) in this set were randomly distributed in $[c_{ij}-10, c_{ij}+10]$, such that $p_{ij}\ge1$. Profits and costs of items were integers. For each test problem the total cost $b$ was calculated as for Experiment 1.\\
The results  for problem set $(wco, 20, 20)$ are given in Table 4.

In the case of uncorrelated data instances, the $BISSA$ algorithm was able to find approximate solutions (and profit values) to problems with 10000 binary variables in reasonable time. The relative difference between profit values of exact and approximate solutions are small for each of the test problems. Upper bounds found by the $BISSA$ algorithm are almost the same as upper bounds found by the greedy algorithm for $(MCKP)$. Even for the problem set $(unc, 1000, 10)$ number of $(BS(\lambda))$ problems solved by the $BISSA$ algorithm is relatively small in regards to number of decision variables.

 In the case of weakly correlated data instances, the $BISSA$ algorithm solved problems with 400 binary variables in reasonable time. The relative difference between profit values of exact and approximate solutions is, in average, greater than for uncorrelated test problems. As one can see in Table 4, upper bounds found by the $BISSA$ algorithm are almost the same as upper bounds found by the greedy algorithm for $(MCKP)$. The reason why the
$BISSA$ algorithm solves weakly correlated instances with a significantly smaller number of variables than for uncorrelated ones in reasonable time is as follows. In line 24 of the $BISSA$ algorithm, in order to find an element $\hat{x}$, we have to go through 
the solution set $S$ to the problem $(BS(\lambda))$ (the complete scan of set $S$ according to values of the second objective function of problem $(BP)$). For weakly correlated data instances the cardinality of the set $S$ may be large even for problems belonging to class $(wco, 30, 30)$. We conducted experiments for problem class $(wco, 30, 30)$. For the most difficult test problem in this class, the cardinality of solution set $S$ to the problem $(BS(\lambda))$ was 199,065,600. For greater weakly correlated problems that number may be even larger.

\begin{table}
\caption{Obtained results for Experiment 1, problem set $(unc, 10, 1000)$.}
\label{table1}
	\begin{tabular}{cccccccccc}
		\hline\noalign{\smallskip}
		1 & 2 & 3 & 4 & 5 & 6 & 7 & 8 & 9 & 10\\ 
		\noalign{\smallskip}\hline\noalign{\smallskip}
		1&	99873&	99849&	24&	0.024&	99887.011&	38.011&	0.038&	99887&	7 \\
		2&	99894&	99889&	5&	0.005&	99899.061&	10.061&	0.010&	99899&	8 \\
		3&	99861&	99861&	0&	0.000&	99866.141&	5.141&	0.005&	99866&	7 \\
		4&	99832&	99832&	0&	0.000&	99836.262&	4.262&	0.004&	99836&	6 \\
		5&	99854&	99854&	0&	0.000&	99856.485&	2.485&	0.002&	99856&	6 \\
		6&	99827&	99808&	19&	0.019&	99835.986&	27.986&	0.028&	99835&	6 \\
		7&	99860&	99841&	19&	0.019&	99863.302&	22.302&	0.022&	99863&	6 \\
		8&	99883&	99883&	0&	0.000&	99895.311&	12.311&	0.012&	99895&	6 \\
		9&	99881&	99881&	0&	0.000&	99883.419&	2.419&	0.002&	99883&	7 \\
		10&	99702&	99702&	0&	0.000&	99724.825&	22.825&	0.023&	99724&	6 \\
		\noalign{\smallskip}\hline
	\end{tabular}
\end{table}

\begin{table}
\caption{Obtained results for Experiment 1, problem set $(unc, 100, 100)$.}
\label{table2}
	\begin{tabular}{cccccccccc}
		\hline\noalign{\smallskip}
		1 & 2 & 3 & 4 & 5 & 6 & 7 & 8 & 9 & 10\\
		\noalign{\smallskip}\hline\noalign{\smallskip}
		1&	983045&	982946&	99&	0.010&	983059.387&	113.387&	0.012&	983059&	10\\ 
		2&	980483&	980433&	50&	0.005&	980492.589&	59.589&	0.006&	980492&	11\\ 
		3&	984106&	983999&	107&	0.011&	984130.851&	131.851&	0.013&	984130&	8\\ 
		4&	982980&	982684&	296&	0.030&	983021.172&	337.172&	0.034&	983021&	10\\
		5&	981421&	981421&	0&	0.000&	981426.965&	5.965&	0.001&	981426&	10\\ 
		6&	983059&	982968&	91&	0.009&	983080.841&	112.841&	0.011&	983080&	10\\
		7&	984059&	984001&	58&	0.006&	984071.849&	70.849&	0.007&	984071&	10\\ 
		8&	987210&	987158&	52&	0.005&	987228.022&	70.022&	0.007&	987228&	9\\
		9&	980999&	980944&	55&	0.006&	981035.911&	91.911&	0.009&	981035&	9\\
		10&	982142&	982060&	82&	0.008&	982176.615&	116.615&	0.012&	982176&	9\\
		\noalign{\smallskip}\hline
\end{tabular}
\end{table}

\begin{table}
\caption{Obtained results for Experiment 1, problem set $(unc, 1000, 10)$.}
\label{table3}
\setlength{\tabcolsep}{0.56em} 
\begin{tabular}{cccccccccc}
		\hline\noalign{\smallskip}
		1 & 2 & 3 & 4 & 5 & 6 & 7 & 8 & 9 & 10\\
		\noalign{\smallskip}\hline\noalign{\smallskip}
		1&8421950&8420352&1598&0.019&8421964.411&1612.411&0.019&8421964&12\\
		2&8770359&8768966&1393&0.016&8770370.988&1404.988&0.016&8770370&11\\
		3&8959068&8958820&248&0.003&8959085.071&265.071&0.003&8959085&12\\
		4&8848233&8847518&715&0.008&8848270.055&752.055&0.008&8848270&11\\
		5&8807777&8806990&787&0.009&8807787.139&797.139&0.009&8807787&12\\ 
		6&8881946&8881649&297&0.003&8881976.338&327.338&0.004&8881976&11\\
		7&8927815&8927065&750&0.008&8927826.311&761.311&0.009&8927826&13\\
		8&8742270&8740874&1396&0.016&8742284.668&1410.668&0.016&8742284&12\\
		9&8693221&8690669&2552&0.029&8693245.349&2576.349&0.030&8693245&12\\
		10&8411809&8411350&459&0.005&8411859.566&509.566&0.006&8411859&12\\
		\noalign{\smallskip}\hline
	\end{tabular}
\end{table}

\begin{table}
\caption{Obtained results for Experiment 2, problem set $(wco, 20, 20)$.}
\label{table4}
\setlength{\tabcolsep}{0.66em} 
	\begin{tabular}{cccccccccc}
		\hline\noalign{\smallskip}
		1 & 2 & 3 & 4 & 5 & 6 & 7 & 8 & 9 & 10\\
		\noalign{\smallskip}\hline\noalign{\smallskip}
		1&113664&113584&80&0.070&113665.988&81.988&0.072&	113665&8\\
		2&102060&102049&11&0.011&102061.000&12.000&0.012&	102061&5\\
		3&91399&89864&1535&1.679&91400.223&1536.223&1.681&	91400&8\\
		4&121378&118231&3147&2.593&121380.379&3149.379&2.595&	121380&8\\
		5&100029&96907&3122&3.121&100032.112&3125.112&3.124&	100032&8\\ 
		6&97145&97145&0&0.000&97146.000&1.000&0.001&	97146&5\\ 
		7&103176&97340&5836&5.656&103178.131&5838.131&5.658&	103178&7\\
		8&82942&82832&110&0.133&82944.000&112.000&0.135&	82944&5\\ 
		9&86132&86132&0&0.000&86132.000&0.000&0.000&		86132&6\\ 
		10&80322&80194&128&0.159&80325.000&131.000&0.163&	80325&6\\ 
		\noalign{\smallskip}\hline\noalign{\smallskip}
	\end{tabular}
\end{table}

\section{Conclusions and future works}

 A new approximate method of solving multiple-choice knapsack problems by replacing the budget constraint with the second objective function has been presented. Such a relaxation of the original problem allows to the smart scanning of the decision space by quick solving of the binary linear optimization problem (it is possible by the decomposition of this problem to independently solved easy subproblems).
Let us note that our method can also be used for finding an upper bound for the multi-dimensional multiple-choice knapsack problem $(MMCKP)$ via the relaxation
obtained by summing up all the linear inequality constraints \cite{Akbar2006}. 

The method can be compared to greedy algorithm for multiple-choice knapsack problems which also finds, in general, an approximate solution and an upper bound.

Two preliminary computational experiments have been conducted to check how the proposed algorithm behaves for simple to solve (uncorrelated) instances and hard to solve (weakly correlated) instances.
The results have been compared to results obtained by the exact state-of-the-art algorithm for multiple-choice knapsack problems \cite{Pisinger1995}.
For weakly correlated problems, the number of solution outcomes which  have to be checked in order to derive the triangle of uncertainty (so also an approximate solution to the problem and its 
upper bound) grows fast with the size of the problem. Therefore, for weakly correlated problems we are able to solve smaller problem instances in reasonable time than for uncorrelated problem instances.

It is worth underlying  that in the proposed method profits and costs of items as well as total cost can be real numbers. It could be of value when one wants to solve multiple-choice knapsack problems without changing real numbers into integers (as one has to do for dynamic programming methods). 

Further work will cover investigations of how the algorithm behaves for weekly and strongly correlated instances as well as on the issue
 of finding a better solution by a smart "scanning"
of the triangle of uncertainty.

\bibliographystyle{spmpsci}
\bibliography{bibliografia}

\begin{thebibliography}{10}
\providecommand{\url}[1]{{#1}}
\providecommand{\urlprefix}{URL }
\expandafter\ifx\csname urlstyle\endcsname\relax
  \providecommand{\doi}[1]{DOI~\discretionary{}{}{}#1}\else
  \providecommand{\doi}{DOI~\discretionary{}{}{}\begingroup
  \urlstyle{rm}\Url}\fi

\bibitem{Akbar2006}
Akbar, M.M., Rahman, M.S., Kaykobad, M., Manning, E.G., Shoja, G.C.: Solving
  the multidimensional multiple-choice knapsack problem by constructing convex
  hulls.
\newblock Computers \& Operations Research \textbf{33}(5), 1259 -- 1273 (2006).
\newblock \doi{http://dx.doi.org/10.1016/j.cor.2004.09.016}.
\newblock
  \urlprefix\url{http://www.sciencedirect.com/science/article/pii/S0305054804002370}

\bibitem{Chen2014}
Chen, Y., Hao, J.K.: A “reduce and solve” approach for the multiple-choice
  multidimensional knapsack problem.
\newblock European Journal of Operational Research \textbf{239}(2), 313 -- 322
  (2014).
\newblock \doi{http://dx.doi.org/10.1016/j.ejor.2014.05.025}.
\newblock
  \urlprefix\url{http://www.sciencedirect.com/science/article/pii/S0377221714004482}

\bibitem{Cherfi2010}
Cherfi, N., Hifi, M.: A column generation method for the multiple-choice
  multi-dimensional knapsack problem.
\newblock Computational Optimization and Applications \textbf{46}(1), 51--73
  (2010).
\newblock \doi{10.1007/s10589-008-9184-7}.
\newblock \urlprefix\url{http://dx.doi.org/10.1007/s10589-008-9184-7}

\bibitem{Dudzinski1984}
Dudzinski, K., Walukiewicz, S.: A fast algorithm for the linear multiple-choice
  knapsack problem.
\newblock Operations Research Letters \textbf{3}(4), 205 -- 209 (1984).
\newblock \doi{http://dx.doi.org/10.1016/0167-6377(84)90027-0}.
\newblock
  \urlprefix\url{http://www.sciencedirect.com/science/article/pii/0167637784900270}

\bibitem{Dudzinski1987}
Dudzinski, K., Walukiewicz, S.: Exact methods for the knapsack problem and its
  generalizations.
\newblock European Journal of Operational Research \textbf{28}(1), 3 -- 21
  (1987).
\newblock \doi{http://dx.doi.org/10.1016/0377-2217(87)90165-2}.
\newblock
  \urlprefix\url{http://www.sciencedirect.com/science/article/pii/0377221787901652}

\bibitem{Dyer1984a}
Dyer, M., Kayal, N., Walker, J.: A branch and bound algorithm for solving the
  multiple-choice knapsack problem.
\newblock Journal of Computational and Applied Mathematics \textbf{11}(2), 231
  -- 249 (1984).
\newblock \doi{http://dx.doi.org/10.1016/0377-0427(84)90023-2}.
\newblock
  \urlprefix\url{http://www.sciencedirect.com/science/article/pii/0377042784900232}

\bibitem{Dyer1995}
Dyer, M., Riha, W., Walker, J.: A hybrid dynamic programming/branch-and-bound
  algorithm for the multiple-choice knapsack problem.
\newblock Journal of Computational and Applied Mathematics \textbf{58}(1), 43
  -- 54 (1995).
\newblock \doi{http://dx.doi.org/10.1016/0377-0427(93)E0264-M}.
\newblock
  \urlprefix\url{http://www.sciencedirect.com/science/article/pii/0377042793E0264M}

\bibitem{Dyer1984}
Dyer, M.E.: An {O(n)} algorithm for the multiple-choice knapsack linear
  program.
\newblock Mathematical Programming \textbf{29}(1), 57--63 (1984).
\newblock \doi{10.1007/BF02591729}.
\newblock \urlprefix\url{http://dx.doi.org/10.1007/BF02591729}

\bibitem{Ehrgott2005}
Ehrgott, M.: Multicriteria Optimization.
\newblock Springer-Verlag Berlin Heidelberg (2005).
\newblock \urlprefix\url{http://www.springer.com/gp/book/9783540213987}

\bibitem{Gao2016}
Gao, C., Lu, G., Yao, X., Li, J.: An iterative pseudo-gap enumeration approach
  for the multidimensional multiple-choice knapsack problem.
\newblock European Journal of Operational Research pp.~-- (2016).
\newblock \doi{http://dx.doi.org/10.1016/j.ejor.2016.11.042}.
\newblock
  \urlprefix\url{//www.sciencedirect.com/science/article/pii/S0377221716309675}

\bibitem{Han2010}
Han, B., Leblet, J., Simon, G.: Hard multidimensional multiple choice knapsack
  problems, an empirical study.
\newblock Computers \& Operations Research \textbf{37}(1), 172 -- 181 (2010).
\newblock \doi{http://dx.doi.org/10.1016/j.cor.2009.04.006}.
\newblock
  \urlprefix\url{http://www.sciencedirect.com/science/article/pii/S0305054809001166}

\bibitem{Hifi2004}
Hifi, M., Michrafy, M., Sbihi, A.: Heuristic algorithms for the multiple-choice
  multidimensional knapsack problem.
\newblock Journal of the Operational Research Society \textbf{55}(12),
  1323--1332 (2004).
\newblock \doi{10.1057/palgrave.jors.2601796}.
\newblock \urlprefix\url{http://dx.doi.org/10.1057/palgrave.jors.2601796}

\bibitem{Kellerer2004}
Kellerer, H., Pferschy, U., Pisinger, D.: Knapsack Problems.
\newblock Springer (2004).
\newblock \urlprefix\url{https://books.google.pl/books?id=u5DB7gck08YC}

\bibitem{Khan1998}
Khan, M.S.: Quality adaptation in a multisession multimedia system: Model,
  algorithms, and architecture.
\newblock Ph.D. thesis, University of Victoria, Victoria, B.C., Canada, Canada
  (1998).
\newblock AAINQ36645

\bibitem{Klamroth2007}
Klamroth, K., Tind, J.: Constrained optimization using multiple objective
  programming.
\newblock Journal of Global Optimization \textbf{37}(3), 325--355 (2007).
\newblock \doi{10.1007/s10898-006-9052-x}.
\newblock \urlprefix\url{http://dx.doi.org/10.1007/s10898-006-9052-x}

\bibitem{Kwong2010}
Kwong, C., Mu, L., Tang, J., Luo, X.: Optimization of software components
  selection for component-based software system development.
\newblock Computers \& Industrial Engineering \textbf{58}(4), 618 -- 624
  (2010).
\newblock \doi{http://dx.doi.org/10.1016/j.cie.2010.01.003}.
\newblock
  \urlprefix\url{http://www.sciencedirect.com/science/article/pii/S0360835210000057}

\bibitem{Lee1999}
Lee, C., Lehoczky, J., Rajkumar, R.R., Siewiorek, D.: On quality of service
  optimization with discrete {QoS} options.
\newblock In: In Proceedings of the IEEE Real-time Technology and Applications
  Symposium, pp. 276--286 (1999)

\bibitem{Martello2000}
Martello, S., Pisinger, D., Toth, P.: New trends in exact algorithms for the
  0–1 knapsack problem.
\newblock European Journal of Operational Research \textbf{123}(2), 325 -- 332
  (2000).
\newblock \doi{http://dx.doi.org/10.1016/S0377-2217(99)00260-X}.
\newblock
  \urlprefix\url{http://www.sciencedirect.com/science/article/pii/S037722179900260X}

\bibitem{Martello1990}
Martello, S., Toth, P.: Knapsack Problems: Algorithms and Computer
  Implementations.
\newblock John Wiley \&amp; Sons, Inc., New York, NY, USA (1990)

\bibitem{Miettinen1999}
Miettinen, K.: Nonlinear Multiobjective Optimization.
\newblock Kluwer Academic Publishers (1999).
\newblock \doi{10.1007/978-1-4615-5563-6}.
\newblock \urlprefix\url{http://users.jyu.fi/~miettine/book/}

\bibitem{Nauss1978}
Nauss, R.M.: The 0–1 knapsack problem with multiple choice constraints.
\newblock European Journal of Operational Research \textbf{2}(2), 125 -- 131
  (1978).
\newblock \doi{http://dx.doi.org/10.1016/0377-2217(78)90108-X}.
\newblock
  \urlprefix\url{http://www.sciencedirect.com/science/article/pii/037722177890108X}

\bibitem{Pisinger1995}
Pisinger, D.: A minimal algorithm for the multiple-choice knapsack problem.
\newblock European Journal of Operational Research \textbf{83}(2), 394 -- 410
  (1995).
\newblock \doi{http://dx.doi.org/10.1016/0377-2217(95)00015-I}.
\newblock
  \urlprefix\url{http://www.sciencedirect.com/science/article/pii/037722179500015I}.
\newblock \{EURO\} Summer Institute Combinatorial Optimization

\bibitem{Pisinger1995code}
Pisinger, D.: Program code in {C}.
\newblock \url{http://www.diku.dk/~pisinger/minknap.c} (1995).
\newblock (downloaded in 2016)

\bibitem{Pisinger2001}
Pisinger, D.: Budgeting with bounded multiple-choice constraints.
\newblock European Journal of Operational Research \textbf{129}(3), 471 -- 480
  (2001).
\newblock \doi{http://dx.doi.org/10.1016/S0377-2217(99)00451-8}.
\newblock
  \urlprefix\url{http://www.sciencedirect.com/science/article/pii/S0377221799004518}

\bibitem{Pyzel2012}
Pyzel, P.: Propozycja metody oceny efektywnosci systemow {MIS}.
\newblock In: A.~Myslinski (ed.) Techniki Informacyjne - Teoria i Zastosowania,
  Wybrane Problemy, vol. 2(14), pp. 59--70. Instytut Badan Systemowych PAN,
  Warszawa (2012)

\bibitem{Sbihi2007}
Sbihi, A.: A best first search exact algorithm for the multiple-choice
  multidimensional knapsack problem.
\newblock Journal of Combinatorial Optimization \textbf{13}(4), 337--351
  (2007).
\newblock \doi{10.1007/s10878-006-9035-3}.
\newblock \urlprefix\url{http://dx.doi.org/10.1007/s10878-006-9035-3}

\bibitem{Sinha1979}
Sinha, P., Zoltners, A.A.: The multiple-choice knapsack problem.
\newblock Operations Research \textbf{27}(3), 503--515 (1979).
\newblock \doi{10.1287/opre.27.3.503}.
\newblock \urlprefix\url{http://dx.doi.org/10.1287/opre.27.3.503}

\bibitem{Zemel1984}
Zemel, E.: An {O(n)} algorithm for the linear multiple choice knapsack problem
  and related problems.
\newblock Information Processing Letters \textbf{18}(3), 123 -- 128 (1984).
\newblock \doi{http://dx.doi.org/10.1016/0020-0190(84)90014-0}.
\newblock
  \urlprefix\url{http://www.sciencedirect.com/science/article/pii/0020019084900140}

\bibitem{Zhong2010}
Zhong, T., Young, R.: Multiple choice knapsack problem: Example of planning
  choice in transportation.
\newblock Evaluation and Program Planning \textbf{33}(2), 128 -- 137 (2010).
\newblock \doi{http://dx.doi.org/10.1016/j.evalprogplan.2009.06.007}.
\newblock
  \urlprefix\url{http://www.sciencedirect.com/science/article/pii/S014971890900041X}.
\newblock Challenges in Evaluation of Environmental Education Programs and
  Policies

\end{thebibliography}

\end{document}